\documentclass[12pt]{amsart}
\usepackage{amsmath,amsthm,latexsym,amscd,amsbsy,amssymb,amsfonts,fleqn,leqno,
euscript, pb-diagram}

\numberwithin{equation}{section}

\newcommand{\U}{\mathcal U}
\newcommand{\V}{\mathcal V}

\newcommand{\I}{\mathbb I}

\newcommand{\e}{\varepsilon}
\newcommand{\diam}{\mathrm{diam}}
\newcommand{\mesh}{\mathrm{mesh}}

\newcommand{\edim}{\mbox{\rm e-dim}}
\newcommand{\erdim}{\mbox{\rm e-rdim}}
\newcommand{\rdim}{\mbox{\rm rdim}}

\newtheorem{thm}{Theorem}[section]
\newtheorem{pro}[thm]{Proposition}
\newtheorem{lem}[thm]{Lemma}

\begin{document}


\title[Parametric Bing and Krasinkiewicz maps: revisited]
{Parametric Bing and Krasinkiewicz maps: revisited}

\author{Vesko  Valov}
\address{Department of Computer Science and Mathematics, Nipissing University,
100 College Drive, P.O. Box 5002, North Bay, ON, P1B 8L7, Canada}
\email{veskov@nipissingu.ca}
\thanks{The author was partially supported by NSERC Grant 261914-08.}




\begin{abstract}
Let $M$ be a complete metric $ANR$-space such that for any metric
compactum $K$ the function space $C(K,M)$ contains a dense set of
Bing (resp., Krasinkiewicz) maps. It is shown that $M$ has the
following property: If $f\colon X\to Y$ is a perfect surjection
between metric spaces, then $C(X,M)$ with the source limitation
topology contains a dense $G_\delta$-subset of maps $g$ such that
all restrictions $g|f^{-1}(y)$, $y\in Y$, are Bing (resp.,
Krasinkiewicz) maps. We apply the above result to establish some
mapping theorems for extensional dimension.
\end{abstract}

\maketitle

\markboth{}{Bing and Krasinkiewicz maps}



\section{Introduction}
All spaces in the paper are assumed to be metrizable and all maps
continuous. By $C(X,M)$ we denote all maps from $X$ into $M$.
Unless stated otherwise, all function spaces are endowed with the
source limitation topology.

In this paper we extend some results from \cite{mv} and \cite{vv}
concerning parametric Bing and Krasinkiewicz maps. A space $M$ is
said to be a {\it Krasinkiewicz space} \cite{mv} if for any
compactum $X$ the function space $C(X,M)$ contains a dense subset of
Krasinkiewicz maps. Here, a map $g\colon X\to M$, where $X$ is
compact, is said to be Krasinkiewicz \cite{ll} if every continuum in
$X$ is either contained in a fiber of $g$ or contains a component of
a fiber of $g$. The class of Krasinkiewicz spaces contains all
Euclidean manifolds and manifolds modeled on Menger or N\"{o}beling
spaces, all polyhedra (not necessarily compact), as well as all
cones with compact bases (see \cite{kr1}, \cite{ll}, \cite{m},
\cite{m1}, \cite{mv}).

Our result concerning parametric Krasinkiewicz maps is as follows:

\begin{thm}
Let $M$ be a complete Krasinkiewicz $ANR$-space and $f\colon X\to Y$
a perfect surjection between metric spaces. Then $C(X,M)$ contains a
dense $G_\delta$-set of maps $g$ such that all restrictions
$g|f^{-1}(y)$, $y\in Y$, are Krasinkiewicz maps.
\end{thm}

Theorem 1.1  was established in \cite{mv} in the case $Y$ is
strongly countable dimensional or $M$ is a closed convex subset of a
Banach space and $Y$ a $C$-space.

The second part of the paper is devoted to Bing maps. Recall that a
map $f$ between compact spaces is said to be a {\em Bing map}
\cite{lev} provided all fibers of $f$ are Bing spaces. Here, a
compactum is a {\em Bing space} if each of its subcontinua are
indecomposable. Following Krasinkiewicz \cite{kr},  we say that a
space $M$ is a {\em free space} if for any compactum $X$ the
function space $C(X,M)$ contains a dense subset consisting of Bing
maps. The class of free spaces is quite large, it contains all
$n$-dimensional manifolds ($n\geq 1$) \cite{kr}, the unit interval
\cite{lev}, all locally finite polyhedrons \cite{st}, all manifolds
modeled on the Menger cube $M_{2n+1}^n$ or the N\"{o}beling space
$N_{2n+1}^n$ \cite{st}, as well as all $1$-dimensional locally
connected continua \cite{st}.

Our second result is the following theorem concerning parametric
Bing maps.

\begin{thm}
Let $f\colon X\to Y$ be a perfect surjection between metric spaces.
Then, for every complete $ANR$ free space $M$ the function space
$C(X,M)$ contains a dense $G_\delta$-set of maps $g$ such that all
restrictions $g|f^{-1}(y)$, $y\in Y$, are Bing maps.
\end{thm}

Theorem 1.2  was established in \cite{vv} in the case $Y$ is
strongly countable dimensional or $M$ is a closed convex subset of a
Banach space and $Y$ a $C$-space.

Theorem 1.2 is applied in the last Section 4 to show that some
mapping theorems for extensional dimension established in \cite{ll}
in the realm of compact metric spaces remain valid for general
metric spaces.

The function space $C(X,M)$ appearing in this paper is endowed with
the source limitation topology whose neighborhood base at a given
function $f\in C(X,M)$ consists of the sets
$$B_\rho(f,\e)=\{g\in C(X,M):\rho(g,f)<\e\},$$ where $\rho$ is a
fixed compatible metric on $M$ and $\e:X\to(0,1]$ runs over
continuous positive functions on $X$. The symbol $\rho(f,g)<\e$
means that $\rho(f(x),g(x))<\e(x)$ for all $x\in X$. It is well know
that for metrizable spaces $X$ this topology doesn't depend on the
metric $\rho$ and it has the Baire property provided $M$ is
completely metrizable.

\section{Parametric Krasinkiewicz maps}

This section contains the proof of Theorem 1.1. We fix a complete
Krasinkiewicz $ANR$-space $M$. Then, by \cite[Lemma 2.1]{vv}, $M$
admits a complete metric $\rho$ generating its topology and having
the following extension property:

\begin{itemize}
\item
If $P$ is a paracompact space, $A\subset P$ a closed set and
$\phi\colon A\to M$ is a map, then for every continuous function
$\alpha\colon P\to (0,1]$ and every map $h\colon A\to M$ with
$\rho(h(z),\phi(z))<\alpha(z)/8$ for all $z\in A$, there exists a
map $\bar{h}\colon P\to M$ extending $h$ such that
$\rho(\bar{h}(z),\phi(z))<\alpha(z)$ for all $z\in P$.
\end{itemize}

Suppose $f\colon X\to Y$ is a perfect surjective map between metric
spaces and $d$ a metric generating the topology of $X$. For every
$A\subset X$ and $\delta>0$ let $B(A,\delta)=\{x\in X:
d(x,A)<\delta\}$. If $y\in Y$ and $m,n\geq 1$, then following the
notations from \cite{mv}, we denote by $\mathcal{K}_f(m,n,y)$ the
set of all maps $g\in C(X,M)$ such that:

\begin{itemize}
\item For each subcontinuum $L\subset f^{-1}(y)$ with diam$g(L)\ge 1/n$
there exists $x\in L$ such that $C(x,g|f^{-1}(y))\subset B(L,1/m)$.
Here, $g|f^{-1}(y)$ is the restriction of $g$ over $f^{-1}(y)$ and
$C(x,g|f^{-1}(y))$ denotes the component of the fiber
$g^{-1}(g(x))\cap f^{-1}(y)$ of $g|f^{-1}(y)$ containing $x$.
\end{itemize}
For a closed set $H\subset Y$ let $\mathcal{K}_f(m,n,H)$ be the
intersection of all $\mathcal{K}_f(m,n,y)$, $y\in H$. We also denote
by $\mathcal K_f(H)$ the set of all maps $g\in C(X,M)$ such that
$g|f^{-1}(y)$ is a Krasinkiewicz map for every $y\in H$. It was
established in \cite{mv} that each $\mathcal{K}_f(m,n,H)$ is open in
$C(X,M)$ and $\mathcal K_f(H)=\bigcap_{m,n \in \mathbb{N}}\mathcal
K_f(m,n,H)$.

Below $\I^k$ is the $k$-dimensional cube and $\mathbb S^{k-1}$ its
boundary.
\begin{lem}\label{compact}
Let $\pi\colon Z\to\I^k$, where $Z$ is a metric compactum and $k\geq
1$. Suppose $g_0\in C(Z,M)$ with $g_0\in\mathcal K_\pi(\mathbb
S^{k-1})$. Then for every $\epsilon>0$ there exists a map $g\in
C(Z,M)$ such that $g\in\mathcal K_\pi(\I^k)$, $g$ is
$\epsilon$-homotopic to $g_0$ and $g|\pi^{-1}(\mathbb
S^{k-1})=g_0|\pi^{-1}(\mathbb S^{k-1})$.
\end{lem}

\begin{proof}
Since $M$ is an $ANR$, any two sufficiently close maps from $X$ into
$M$ are homotopic. Hence, we need to prove that for every
$\epsilon>0$ there exists a map $g\in\mathcal K_\pi(\I^k)$ which is
$\epsilon$-close to $g_0$ and $g|\Omega=g_0|\Omega$, where
$\Omega=\pi^{-1}(\mathbb S^{k-1})$. According to \cite[Proposition
2.4]{mv}, each set $T(m,n)=\mathcal K_\pi(m,n,\I^k)$ is open in
$C(Z,M)$.

\textit{Claim $1$. Let $g\in C(Z,M)$ with $g|\Omega=g_0|\Omega$.
Then for any  $m,n\geq 1$ and $\delta>0$ there exists $h\in T(m,n)$
such that $h|\Omega=g_0|\Omega$ and $h$ is $\delta$-close to $g$.}

We fix $m,n\geq 1$. Since $g|\Omega=g_0|\Omega$, $g\in\mathcal
K_\pi(m,n,y)$ for every $y\in\mathbb S^{k-1}$. Then, by \cite[Lemma
2.3]{mv}, there exist a neighborhood $V_y$ of $y$ in $\I^k$ and
$\delta_y>0$ satisfying the following condition: For every $y'\in
V_y$ and $g'\in C(Z,M)$ with $\rho(g'(x),g(x))<\delta_y$ for all
$x\in\pi^{-1}(y')$ we have $g'\in\mathcal K_\pi(m,n,y')$. Choose
finitely many $y_i\in\mathbb S^{k-1}$, $i\leq s$, such that
$\displaystyle V=\bigcup_{i\leq s}V_{y_i}$  covers $\mathbb
S^{k-1}$. Let $F=\I^k\backslash V$ and
$\displaystyle\eta=\min\{\delta,\delta_{y_i}:i\leq s\}$. Because $M$
is a Krasinkiewicz space,  there is a map $g_1\in\mathcal
K_\pi(\I^k)$ which is $(\eta/8)$-close to $g$. Then the map
$g_2\colon \Omega\cup\pi^{-1}(F)\to M$, $g_2|\Omega=g|\Omega$ and
$g_2|\pi^{-1}(F)=g_1|\pi^{-1}(F)$, is $(\eta/8)$-close to
$g|\Omega\cup\pi^{-1}(F)$. According to the extension property of
$(M,\rho)$, $g_2$ can be extended to a map $h\colon Z\to M$ which is
$\eta$-close to $g$. We have $h\in\mathcal K_\pi(m,n,y)$ for all
$y\in\I^k$. Indeed, this follows from the choice of $V_{y_i}$ and
$\delta_i$, $i\leq s$, when $y\in V$, and from the fact that
$g_1\in\mathcal K_\pi(\I^k)$ when $y\in F$. Hence, $h\in T(m,n)$
which completes the proof of the claim.

Now, we arrange all $T(m,n)$, $m,n\geq 1$, as a sequence
$\{T_j:j\geq 1\}$. Using the above claim and the openness of each
$T_j$ in $C(Z,M)$, we construct by induction a sequence of positive
numbers $\{\epsilon_j\}_{j\geq 0}$ and a sequence of maps
$\{g_j\}_{j\geq 1}\subset C(Z,M)$ satisfying the following
conditions for any $j\geq 0$:

\begin{itemize}
\item $\epsilon_0=\epsilon/3$ and $\epsilon_{j+1}\leq\epsilon_j/2$;
\item $g_{j+1}\in B_\rho(g_j,\epsilon_j)\cap T_{j+1}$;
\item $B_\rho(g_{j+1},3\epsilon_{j+1})\subset T_{j+1}$;
\item $g_{j+1}|\Omega=g_0|\Omega$.
\end{itemize}

The sequence $\{g_j\}_{j\geq 1}$ converges uniformly to a map $g\in
C(Z,M)$. It follows from the construction that $g|\Omega=g_0|\Omega$
and $g$ is $\epsilon$-close to $g_0$. Moreover,
$\displaystyle\rho(g(x),g_j(x))\leq\sum_{i=j}^{\infty}\epsilon_i\leq
2\epsilon_j$ for any $j\geq 1$ and $x\in Z$. Therefore, $g\in T_j$,
$j\geq 1$. Since, by \cite[Proposition 2.1]{mv}, $\mathcal
K_\pi(\I^k)$ is the intersection of all $T_j$, $g$ is as required.
\end{proof}

Next step is to prove that if $f\colon N\to L$ is a perfect $PL$-map
between simplicial complexes with the $CW$-topology, then $\mathcal
K_f(L)$ is dense in $C(N,M)$. Recall that $f$ is a $PL$-map if
$f(\sigma)$ is contained in a simplex of $L$ and $f$ is linear on
$\sigma$ for every simplex $\sigma$ of $N$. In general,  $N$ and $L$
are  not metrizable,  but all of their compact subsets are
metrizable.

\begin{lem}\label{simcomplex}
Let $N,L$ be simplicial complexes and $f\colon N\to L$ a perfect
$PL$-map. Then $\mathcal K_f(L)$ is dense in $C(N,M)$.
\end{lem}

\begin{proof}
Fix $g\in C(N,M)$ and $\alpha\in C(N,(0,1])$. afollowing the proof
of \cite[Lemma 11.3]{bv}, we are going to find a map $h\in\mathcal
K_f(L)$ which is $\alpha$-close to $g$. To this end, let $L^{(i)}$,
$i\geq 0$, denote the $i$-dimensional skeleton of $L$. We put
$L^{(-1)}=\varnothing$ and $h_{-1}=g$. Construct inductively a
sequence $(h_i:N\to M)_{i\geq 0}$ of maps such that
\begin{itemize}
\item $h_{i}|f^{-1}(L^{(i-1)})=h_{i-1}|f^{-1}(L^{(i-1)})$;
\item $\displaystyle h_{i}$ is $\displaystyle\frac{\alpha}{2^{i+2}}$-homotopic to $h_{i-1}$;
\item $h_i\in\mathcal K_f(L^{(i)})$ for every $i$.
\end{itemize}

Assuming that the map $h_{i-1}:N\to M$ has been constructed,
consider the complement $L^{(i)}\setminus L^{(i-1)}=\sqcup_{j\in
J_i}\overset{\circ}\sigma_j$, which is the discrete union of open
$i$-dimensional simplexes.  Since $h_{i-1}\in\mathcal
K_f(L^{(i-1)})$, we can apply Lemma~\ref{compact} for every simplex
$\sigma_j\in L^{(i)}$ to find a map $g_j:f^{-1}(\sigma_j)\to M$ such
that
\begin{itemize}
\item $g_j$ coincides with $h_{i-1}$ on the set $f^{-1}(\sigma^{(i-1)}_j)$;
\item $g_j$ is $\displaystyle\frac{\alpha}{2^{i+2}}$-homotopic to
$h_{i-1}$;
\item  all restrictions $g_j|f^{-1}(y)$, $y\in\sigma_j$, are Krasinkiewicz maps.
\end{itemize}

Define a map $\varphi_i:f^{-1}(L^{(i)})\to M$ by the formula
$$\varphi_i(x)=\begin{cases} h_{i-1}(x)&\mbox{if $x\in
f^{-1}(L^{(i-1)})$;}\\ g_j(x)&\mbox{if $x\in f^{-1}(\sigma_j)$.}
\end{cases}$$ It can be shown that $\varphi_i$ is
$\displaystyle\frac{\alpha}{2^{i+2}}$-homotopic to
$h_{i-1}|f^{-1}(L^{(i)})$. By the Homotopy Extension Theorem,
there exists a continuous extension $h_i:N\to M$ of the map $\varphi_i$ which is
$\displaystyle\frac{\alpha}{2^{i+2}}$-homotopic to $h_{i-1}$.
The map $h_i$ satisfies the inductive conditions.

Then the limit map $h=\lim_{i\to\infty}h_i:N\to M$ is well-defined,
continuous and $\alpha$-close to $g$. Finally, since
$h|f^{-1}(L^{(i)})=h_i|f^{-1}(L^{(i)})$ for every $i\geq 0$,
$h\in\mathcal K_f(L)$.
\end{proof}

Now we can complete the proof of Theorem 1.1. We use some arguments
from the proof of \cite[Proposition 3.4]{bv1}.

\begin{pro}\label{general reg}
Suppose $f\colon X\to Y$ is a perfect surjection between metric
spaces. Then  $\mathcal K_f(Y)$ is a dense and $G_\delta$-subset of
$C(X,M)$.
\end{pro}

\begin{proof}
Since $C(X,M)$ has the Baire property and $\mathcal K_f(Y)$ is the
intersection of the open sets $\mathcal K_f(m,n,Y)\subset C(X,M)$,
$m,n\geq 1$, it suffices to show that each $\mathcal K_f(m,n,Y)$ is
dense in $C(X,M)$. Let $\epsilon\in C(X,(0,1])$ and a $g\in C(X,M)$.
We are going to find a map $h\in\mathcal K_f(m,n,Y)$ such that
$\rho(g(x),h(x))<\epsilon(x)$ for all $x\in X$. Since $M$ is an
$ANR$, $g$ can be approximated by simplicially factorizable maps
(i.e., maps $g'\in C(X,M)$ such that $g'=g_2\circ g_1$, where $g_1$
is a map from $X$ into a simplicial complex $L'$ and $g_2\colon
L'\to M$). So, $g$ itself can be assumed to be simplicially
factorizable. Choose a simplicial complex $D$ and maps $g_D\colon
X\to D$, $g^D\colon D\to M$ with $g=g^D\circ g_D$. The metric $\rho$
induces a continuous pseudometric $\rho_D$ on $D$,
$\rho_D(x,y)=\rho(g^D(x),g^D(y))$. Since $D$ is a neighborhood
retract of a locally convex space, we can apply \cite[Lemma 8.1]{bv}
to find an open cover $\U$ of $X$ satisfying the following
condition: if $\alpha\colon X\to K$ is a $\U$-map into a paracompact
space $K$ (i.e., $\alpha^{-1}(\omega)$ refines $\U$ for some open
cover $\omega$ of $K$), then there exists a map $q'\colon G\to D$,
where $G$ is an open neighborhood of $\overline{\alpha(X)}$ in $K$,
such that $g_D$ and $q'\circ\alpha$ are $\epsilon/2$-close with
respect to the pseudometric $\rho_D$. Let $\U_1$ be an open cover of
$X$ refining $\U$ with $\mesh\U_1<1/m$ and $\inf\{\epsilon(x):x\in
U\}>0$ for all $U\in\U_1$.

Next, according to \cite[Theorem 6]{bv}, there exists an open cover
$\V$ of $Y$ such that: for any $\V$-map $\beta\colon Y\to L$ into a
simplicial complex $L$ we can find an $\U_1$-map $\alpha\colon X\to
K$ into a simplicial complex $K$ and a perfect $PL$-map $p\colon
K\to L$ with $\beta\circ f=p\circ\alpha$. We can assume that $\V$ is
locally finite. Take $L$ to be the nerve of the cover $\V$ and
$\beta\colon Y\to L$ the corresponding natural map. Then there are a
simplicial complex $K$ and maps $p$ and $\alpha$ satisfying the
above conditions. Hence, the following diagram is commutative:
$$
\begin{CD}
X@>{\alpha}>>K\cr @V{f}VV @VV{p}V\cr Y@>{\beta}>>L\cr
\end{CD}
$$
Since $K$ is paracompact, the choice of the cover $\U$ guarantees
the existence of a map $\varphi_D\colon G\to D$, where $G\subset K$
is an open neighborhood of $\overline{\alpha(X)}$, such that $g_D$
and $h_D=\varphi_D\circ\alpha$ are $\epsilon/2$-close with respect
to $\rho_D$. Then, according to the definition of $\rho_D$,
$h'=g^D\circ \varphi_D\circ\alpha$ is $\epsilon/2$-close to $g$ with
respect to $\rho$. Replacing the triangulation of $K$ by a suitable
subdivision, we may additionally assume that no simplex of $K$ meets
both $\overline{\alpha(X)}$ and $K\backslash G$. So, the union $N$
of all simplexes $\sigma\in K$ with
$\sigma\cap\overline{\alpha(X)}\neq\varnothing$ is a subcomplex of
$K$ and $N\subset G$. Moreover, since $N$ is closed in $K$,
$p_N=p|N\colon N\to L$ is a perfect map. Therefore, we have the
following commutative diagram, where $\varphi=g^D\circ \varphi_D$:

\begin{picture}(120,95)(-100,0)
\put(30,10){$L$} \put(0,30){$Y$} \put(12,28){\vector(3,-2){18}}
\put(14,14){\small $\beta$} \put(1,70){$X$}
\put(5,66){\vector(0,-1){25}} \put(-1,53){\small $f$}
\put(11,73){\vector(1,0){45}} \put(30,77){\small $h'$}
\put(12,68){\vector(3,-2){18}} \put(15,56){\small $\alpha$}
\put(31,50){$N$} \put(35,46){\vector(0,-1){25}}
 \put(37,33){\small $p_N$}
\put(46,58){\vector(4,3){13}} \put(44,64){\small $\varphi$}
 \put(60,70){$M$}
\end{picture}

Using that $\alpha$ is a $\U_1$-map and
$\inf\{\epsilon(x):x\in U\}>0$ for all $U\in\U_1$, we can
construct a continuous function $\epsilon_1:N\to(0,1]$
with $\epsilon_1\circ\alpha\leq\epsilon$.
By Lemma~\ref{simcomplex}, there exists a map $\varphi_1\in C(N,M)$
which is $\epsilon_1/2$-close to  $\varphi$ and
$\varphi_1\in\mathcal K_{p_N}(L)$. Let $h=\varphi_1\circ\alpha$.
Then $h$ and $\varphi\circ \alpha$ are $\epsilon/2$-close because
$\epsilon_1\circ\alpha\leq\epsilon$. On the other hand,
$\varphi\circ\alpha=h'$ is $\epsilon/2$-close to $g$. Hence, $g$ and
$h$ are $\epsilon$-close.

It remains to show that $h\in\mathcal K_f(m,n,Y)$. To this end, fix
$y\in Y$ and a continuum $P\subset f^{-1}(y)$ with $\diam h(P)\geq
1/n$. Then $P'=\alpha(P)$ is a continuum in $p_N^{-1}(\beta(y))$
which is not contained in any fiber of the restriction map
$\varphi_1|p_N^{-1}(\beta(y))$. Since the last map is Krasinkiewicz,
there exists a point $z^*\in P'$ such that $C'\subset P'$, where
$C'$ is the component of
$p_N^{-1}(\beta(y))\cap\varphi_1^{-1}(\varphi_1(z^*))$ containing
$z^*$ . Choose $x^*\in P$ with $\alpha(x^*)=z^*$ and a component $C$
of $f^{-1}(y)\cap h^{-1}(h(x^*))$ which contains $x^*$. Then
$\alpha(C)$ is a connected subset of
$p_N^{-1}(\beta(y))\cap\varphi_1^{-1}(\varphi_1(z^*))$ meeting $C'$.
Consequently, $\alpha(C)\subset C'\subset\alpha (P)$. Hence,
$\alpha^{-1}(\alpha(x))\cap P\neq\varnothing$ for any $x\in C$.
Finally, since $\mesh\U_1<1/m$ and $\alpha$ is an $\U_1$-map,
$diam\alpha^{-1}(\alpha(x))<1/m$ for all $x\in C$. So, $C\subset
B(P,1/m)$. Therefore, $h\in\mathcal K_f(m,n,y)$ which completes the
proof.
\end{proof}

\section{Parametric Bing maps}

In this section we establish the proof of Theorem 1.2. Everywhere in
the section $(M,\rho)$ is a fixed complete $ANR$ free space having
the extension property from Section 2. Suppose $f\colon X\to Y$ is a
perfect surjection and $d$ is a metric on $X$ generating its
topology. For any $y\in Y$ and $m\geq 1$ let $\mathcal B_f(m,y)$ be
the set of all maps $g\in C(X,M)$ such that:

\begin{itemize}
\item If $z\in M$ and $K_0, K_1$ are two continua in $f^{-1}(y)\cap
g^{-1}(z)$ such that $K_0\cap K_1\neq\varnothing$, then either
$K_0\subset B(K_1,1/m)$ or $K_1\subset B(K_0,1/m)$.
\end{itemize}

For any closed set $H\subset Y$ let $\mathcal B_f(m,H)=\bigcap_{y\in
H}\mathcal B_f(m,y)$ and $\mathcal B_f(H)$ be the set of all $g\in
C(X,M)$ such that $g|f^{-1}(y)$ is a Bing map for any $y\in H$.

\begin{lem}
$\mathcal B_f(H)=\bigcap_{m\geq 1}\mathcal B_f(m,H)$.
\end{lem}

\begin{proof}
It is easily seen that $\mathcal B_f(H)\subset\bigcap_{m\geq
1}\mathcal B_f(m,H)$. Assume $g\in \bigcap_{m\geq 1}\mathcal
B_f(m,H)$ and $y\in H$. Then $g|f^{-1}(y)$ is a Bing map. Indeed,
otherwise there exists $z\in M$ and two non-trivial continua
$K_0,K_1\subset f^{-1}(y)\cap g^{-1}(z)$ with a common point such
that $K_i\backslash K_{1-i}\neq\varnothing$ for $i=0,1$. Let $x_i\in
K_i\backslash K_{1-i}$ and choose $m$ so big that $x_i\not\in
B(K_{1-i},1/m)$, $i=0,1$. So, $K_i\nsubseteq B(K_{1-i},1/m)$ for any
$i=0,1$ which contradicts the fact that $g\in\mathcal B_f(m,y)$.
\end{proof}

\begin{lem}
Each $\mathcal B_f(m,H)$ is open in $C(X,M)$.
\end{lem}

\begin{proof}
We need first the following claim:

\textit{Claim $2$. Let $g\in\mathcal B_f(m,y)$ for some $y\in Y$ and
$m\geq 1$. Then there exists a neighborhood $V_y$ of $y$ in $Y$ and
$\delta_y>0$ such that $y'\in V_y$ and $g'\in C(X,M)$ with
$\rho\big(g'(x),g(x)\big)<\delta_y$ for all $x\in f^{-1}(y')$ yields
$g'\in\mathcal B_f(m,y')$.}

Indeed, otherwise we can find a local base $\{V_k\}_{k \in
\mathbb{N}}$ of neighborhoods of $y$ in $Y$, points $y_k \in V_k$
and maps $g_k \in C(X,M)$ such that $\rho(g_k(x), g(x)) < 1/k$ for
all $x \in f^{-1}(y_k)$ but $g_k\not\in\mathcal B_f(m,y_k)$.
Consequently, for any $k$ there exist $z_k\in M$, two continua
$P_k^0, P_k^1\subset f^{-1}(y_k)\cap g_k^{-1}(z_k)$ having a common
point, and points $x_k^i\in P_k^i$ with $d(x_k^i,P_k^{1-i})\geq
1/m$, $i=0,1$. Since the set $f^{-1}(\{y_k\}_{k \in \mathbb{N}} \cup
\{y\})$ is compact and contains all $P_k^i$, we may assume that each
of the sequences $\{P_k^i\}$ and $\{x_k^i\}$, $i=0,1$, converge,
respectively, to a continuum $P^i\subset f^{-1}(y)$ and a point
$x^i\in P^i$. This implies that $\{z_k\}$ converges to a point $z\in
M$ with $P^i\subset f^{-1}(y)\cap g^{-1}(z)$. Moreover, we have
$P^0\cap P^1\neq\varnothing$ and $d(x^i,P^{1-i})\geq 1/m$ for any
$i=0,1$. The last conditions contradict the fact that $g\in\mathcal
B_f(m,y)$. This completes the proof of the claim.

Now, following the arguments from the proof of \cite[Proposition
2.3]{vv} and using Claim 2 instead of \cite[Lemma 2.2]{vv}, we can
establish that for any closed $H\subset Y$ the set  $\mathcal
B_f(m,H)$ is open in $C(X,M)$.
\end{proof}

The proof of next two lemmas is similar to the proof of Lemma 2.1
and Lemma 2.2, respectively.

\begin{lem}
Let $\pi\colon Z\to\I^k$, where $Z$ is a metric compactum and $k\geq
1$. Suppose $g_0\in C(Z,M)$ with $g_0\in\mathcal B_\pi(\mathbb
S^{k-1})$. Then for every $\epsilon>0$ there exists a map $g\in
C(Z,M)$ such that $g\in\mathcal B_\pi(\I^k)$, $g$ is
$\epsilon$-homotopic to $g_0$ and $g|\pi^{-1}(\mathbb
S^{k-1})=g_0|\pi^{-1}(\mathbb S^{k-1})$.
\end{lem}

\begin{lem}
Let $N,L$ be simplicial complexes and $f\colon N\to L$ a perfect
$PL$-map. Then $\mathcal B_f(L)$ is dense in $C(N,M)$.
\end{lem}

Now we can complete the proof of Theorem 1.2.
\begin{pro}
Suppose $f\colon X\to Y$ is a perfect surjection between metric
spaces. Then  $\mathcal B_f(Y)$ is a dense and $G_\delta$-subset of
$C(X,M)$.
\end{pro}

\begin{proof}
Since $\mathcal B_f(Y)$ is the intersection of the open sets
$\mathcal B_f(m,Y)$, $m\geq 1$, it suffices to show that each
$\mathcal B_f(m,Y)$ is dense in $C(X,M)$. To this end, we repeat the
construction and the notations from the proof of Proposition 2.3.
Using Lemma 3.4 instead of Lemma 2.2, we obtain a map $\varphi_1\in
C(N,M)$ such that $\varphi_1\in\mathcal B_{p_N}(L)$ and $\varphi_1$
is $\epsilon/2$-close to $\varphi$. Let us show that the map
$h=\varphi_1\circ\alpha$ belongs to $\mathcal B_f(m,Y)$. We fix
$y\in Y$ and take two continua $P^0,P^1\subset f^{-1}(y)\cap
h^{-1}(z)$ for some $z\in M$ such that $P^0\cap P^1\neq\varnothing$.
Then $C^0=\alpha(P^0)$ and $C^1=\alpha(P^1)$ are non-empty continua
in $p_N^{-1}(\beta(y))\cap\varphi_1^{-1}(z)$ having a common point.
Since $p_N^{-1}(\beta(y))\cap\varphi_1^{-1}(z)$ is a Bing space, one
of these continua is contained in the other one. Assume that
$C^0\subset C^1$. Then $\alpha^{-1}(\alpha(x))\cap
P^1\neq\varnothing$ for all $x\in P^0$. Because all fibers of
$\alpha$ are of diameter $<1/m$ (recall that $mesh\U_1<1/m$), we
finally obtain that $P^0\subset B(P^1,1/m)$. Hence, $h\in\mathcal
B_f(m,y)$ which completes the proof.
\end{proof}

\section{Some mapping theorems for extensional dimension}

In this section we apply Theorem 1.2 to extend some results from
\cite{ll} which were established for maps between compact metric
spaces. Extension theory, which was first introduced by Dranishnikov
\cite{dr}, is based on the following notion. If $K$ is a
$CW$-complex we say that the extension dimension of a space $X$ does
not exceed $K$ (br., $\edim X\leq K$) provided $K$ is an absolute
extensor for $X$. For example, $\dim X\leq n$ if and only if
$\mathbb{S}^n\in AE(X)$. For a map $f\colon X\to Y$ we write $\edim
f\leq K$ provided $\edim f^{-1}(y)\leq K$ for all $y\in Y$.

We start with the following proposition established in \cite[Theorem
1.2]{l2} for maps between compact metric spaces.

\begin{pro}
Let $K$ be a $CW$-complex and $f\colon X\to Y$ be a perfect
surjection between metric spaces with $\edim f\leq K$. Then there
exists an $F_\sigma$-set $A\subset X$ such that $\edim A\leq K$ and
$\dim f^{-1}(y)\backslash A\leq 0$ for all $y\in Y$.
\end{pro}

\begin{proof}
According to \cite{bp}, there exists a map $g\colon X\to Q$, where
$Q=\I^\omega$ is the Hilbert cube, such that $g$ embeds every fiber
$f^{-1}(y)$, $y\in Y$. Let $g=\triangle_{i=1}^{\infty}g_i$ and
$h_i=f\triangle g_i\colon X\to Y\times\I$, $i\geq 1$. Moreover, we
choose countably many closed intervals $\I_j$ such that every open
subset of $\I$ contains some $\I_j$. By \cite[Lemma 4.1]{tv}, for
every $j$ there exists a $0$-dimensional $F_\sigma$-set $C_j\subset
Y\times\I_j$ such that $C_j\cap (\{y\}\times\I_j)\neq\varnothing$
for every $y\in Y$.  Now, consider the sets $A_{ij}=h_i^{-1}(C_j)$
for all $i,j\geq 1$ and let $A$ be their union. Since $\edim f\leq
K$, $\edim h_i\leq K$ for any $i$. Hence, according to
\cite[Corollary 2.5]{cv}, $\edim A_{ij}\leq K$ for all $i,j$. This
implies $\edim A\leq K$.

It remains to show that $\dim f^{-1}(y)\backslash A\leq 0$ for every
$y\in Y$. Suppose $\dim f^{-1}(y_0)\backslash A>0$ for some $y_0$.
Since $g|f^{-1}(y_0)$ is an embedding, there exists $i$ such that
$\dim g_i(f^{-1}(y_0)\backslash A)>0$. Then
$g_i(f^{-1}(y_0)\backslash A)$  has a non-empty interior in $\I$.
So, $g_i(f^{-1}(y_0)\backslash A)$ contains some $\I_j$. Choose
$t_0\in\I_j$ with $c_0=(y_0,t_0)\in C_j$. Then there exists $x_0\in
f^{-1}(y_0)\backslash A$ such that $g_i(x_0)=t_0$. On the other
hand, $x_0\in h_i^{-1}(c_0)\subset A_{ij}\subset A$, a
contradiction.
\end{proof}

For a set $A\subset X$ we write $\rdim_X A\leq n$ provided $\dim
H\leq n$ for every closed subset $H$ of $X$ which is contained in
$A$. Next theorem is an analogue of Theorem 1.9 from \cite{ll}.

\begin{thm}
Let $K$ be a $CW$-complex and $f\colon X\to Y$ a perfect map. Let
$\tilde{X}=X\times\I$ and define $\tilde{f}\colon\tilde{X}\to Y$ by
$\tilde{f}(x,t)=f(x)$. Consider the following properties:

\begin{enumerate}
\item[(1)] $\edim f\leq K$;
\item[(2)] Almost every map $g\colon\tilde{X}\to\I$ is such that
$\edim\tilde{f}\triangle g\leq K$;
\item[(3)] There exists an $F_\sigma$-subset $A$ of $\tilde{X}$ such
that $\edim A\leq K$ and
$\displaystyle\rdim_{\tilde{X}}\tilde{f}^{-1}(y)\backslash A\leq 0$
for every $y\in Y$.
\item[(4)] There exists a $G_\delta$-subset $B$ of $\tilde{X}$ such
that $\edim B\leq K$ and $dim\tilde{f}^{-1}(y)\backslash B\leq 0$
for every $y\in Y$.
\end{enumerate}

Then $(1)\Rightarrow (2)\Rightarrow (3)$. Moreover, if $K$ is
countable, then $(1)\Rightarrow (2)\Rightarrow (3)\Rightarrow (4)$.
\end{thm}

\begin{proof}
To proof the implication $(1)\Rightarrow (2)$, take a map
$g\colon\tilde{X}\to\I$ such that all restrictions
$g|\tilde{f}^{-1}(y)$, $y\in Y$, are Bing maps. So, any fiber
$(\tilde{f}\triangle g)^{-1}(y,t)=g^{-1}(t)\cap(f^{-1}(y)\times\I)$,
$t\in\I$, $y\in Y$, is a Bing space. So, it does not contain a
non-degenerate interval. This implies that the projection of
$(\tilde{f}\triangle g)^{-1}(y,t)$ onto $f^{-1}(y)$ is a
$0$-dimensional map. Then, by \cite[Theorem 1.2]{du},
$\edim\tilde{f}\triangle g\leq\edim f^{-1}(y)\leq K$. Since $\I$ is
a free space \cite{lev}, by Theorem 1.2, almost every map $g\in
C(\tilde{X},\I)$ is such that the fibers of $\tilde{f}\triangle g$
are Bing spaces. The implication $(1)\Rightarrow (2)$ is
established.

For the implication $(2)\Rightarrow (3)$, we fix a map
$g\colon\tilde{X}\to\I$ such that $\edim\tilde{f}\triangle g\leq K$
and consider the family of all subintervals $\I_j$ of $\I$ with
rational end-points. As in the proof of Theorem 4.1, for every $j$
there exists a $0$-dimensional $F_\sigma$-set $C_j\subset
Y\times\I_j$ such that $(\{y\}\times\I_j)\cap C_j\neq\varnothing$
for every $y\in Y$. Let $C=\bigcup_{j\geq 1}C_j$ and
$A_1=(\tilde{f}\triangle g)^{-1}(C)$. Since $\dim C=0$, it follows
from \cite[Corollary 2.5]{cv} that $\edim A_1\leq K$. On the other
hand, by Theorem 4.1, there exists an $F_\sigma$-subset
$A_2\subset\tilde{X}$ such that $\edim A_2\leq K$ and $\dim
(\tilde{f}\triangle g)^{-1}(y,t)\backslash A_2\leq 0$ for all
$(y,t)\in Y\times\I$. Then $A=A_1\cup A_2$ is an $F_\sigma$-set in
$\tilde{X}$ and $\edim A\leq K$. It remains to show that
$\rdim_{\tilde{X}}\tilde{f}^{-1}(y)\backslash A\leq 0$ for every
$y\in Y$. So, we fix $y_0\in Y$ and a closed set $H\subset\tilde{X}$
which is contained in $\tilde{f}^{-1}(y_0)\backslash A$. Suppose
$\dim H>0$. Then $H$ contains a non-trivial component of
connectedness $T$ (recall that $H$ is compact because so is
$\tilde{f}^{-1}(y_0)$). Since  $(\tilde{f}\triangle
g)^{-1}(y_0,t)\backslash A$ is $0$-dimensional for all $t\in\I$,
$g(T)$ is a non-trivial subinterval of $\I$. So, $g(T)$ contains
some $\I_j$. Consequently,  $T\cap (\tilde{f}\triangle
g)^{-1}(C_j)\neq\varnothing$ which is a contradiction. Hence,
$\rdim_{\tilde{X}}\tilde{f}^{-1}(y_0)\backslash A\leq 0$.

Finally, let us show the implication $(3)\Rightarrow (4)$ provided
$K$ is countable. Suppose $A\subset\tilde{X}$ is an $F_\sigma$-set
such that $\edim A\leq K$ and
$\displaystyle\rdim_{\tilde{X}}\tilde{f}^{-1}(y)\backslash A\leq 0$
for every $y\in Y$. Since $K$ is countable, every metric space of
extension dimension $\leq K$ has a completion with the same
dimension \cite{l1}. This fact easily implies that $A$ can be
enlarged to a $G_\delta$-set $B\subset\tilde{X}$ with $\edim B\leq
K$. Then $\tilde{f}^{-1}(y)\backslash B$ is an $F_\sigma$-subset of
$\tilde{X}$ which is contained in $\tilde{f}^{-1}(y)\backslash A$,
$y\in Y$. Hence, $\dim\tilde{f}^{-1}(y)\backslash B\leq 0$ because
$\rdim_{\tilde{X}}\tilde{f}^{-1}(y)\backslash A\leq 0$.
\end{proof}

For a subset $B\subset X$ and a $CW$-complex $L$ we write $\erdim_X
B\leq L$ provided $\edim F\leq L$ for any closed set $F$ in $X$
which is contained in $B$. Our next result extends Theorem 1.5 from
\cite{ll}.

\begin{thm}
Let $K,L$ be two $CW$-complexes and $f\colon X\to Y$ a perfect
surjection such that $\edim f\leq K$ and $\edim Y\leq L$. Then there
exists a decomposition  $\tilde{X}=X\times\I=A\cup B$ of $\tilde{X}$
such that $\edim A\leq K$ and $\erdim_{\tilde{X}} B\leq L$. If in
addition, $K$ is countable, $B$ can be chosen so that $\edim B\leq
L$. In such a case $\edim\tilde{X}\leq K*L$.
\end{thm}

\begin{proof}
According to Theorem 4.2(3), there exists an $F_\sigma$-set
$A\subset\tilde{X}$ with $\edim A\leq K$ and
$\displaystyle\rdim_{\tilde{X}}\tilde{f}^{-1}(y)\cap B\leq 0$ for
every $y\in Y$,  where $B=\tilde{X}\backslash A$. Then, $f|F$ is a
perfect $0$-dimensional map for every closed set $F\subset\tilde{X}$
which is contained in $B$. Hence, by \cite[Corollary 2.5]{cv},
$\edim F\leq L$. So, $\erdim_{\tilde{X}} B\leq L$.

If $K$ is countable, we apply Theorem 4.2(4) to find a
$G_\delta$-set $A\subset\tilde{X}$ such that $\edim A\leq K$ and
$\tilde{f}|B$ being $0$-dimensional with $B=\tilde{X}\backslash A$.
Let $B=\bigcup_{i\geq 1}B_i$ with each $B_i\subset\tilde{X}$ closed.
Then all restrictions $\tilde{f}|B_i$, $i\geq 1$, are
$0$-dimensional perfect maps. So, applying again \cite[Corollary
2.5]{cv} we obtain $\edim B_i\leq L$, $i\geq 1$. Hence, $\edim B\leq
L$. Finally, according to \cite[Theorem A]{dy}, $\edim\tilde{X}\leq
K*L$.
\end{proof}

Let us also mention the following corollary of Theorem 4.3.

\begin{thm}
Let $K,L$ be connected $CW$-complexes and $f\colon X\to Y$ a perfect
surjection such that $\edim f\leq K$ and $\edim Y\leq L$. If $X$ is
finite-dimensional and $K$ is countable, then $\edim X\leq
K\bigwedge L$.
\end{thm}

\begin{proof}
We can repeat the proof of Theorem 1.6 from \cite{ll}. The only
difference is that we apply now Theorem 4.3 instead of \cite[Theorem
1.5]{ll} and the application of the extensivity criterion of
Dranishnikov \cite{dr1} is replaced by similar results of Dydak
\cite{dy1} concerning general metric spaces.
\end{proof}



\end{document}